\documentclass[12pt,a4paper]{amsart}
\usepackage{a4wide}
\usepackage{amsmath,amsfonts,amssymb,amscd,amsthm,latexsym,mathrsfs,comment}
\usepackage{color}

\theoremstyle{plain}
\newtheorem{Thm}{Theorem}[section]
\newtheorem{Lem}[Thm]{Lemma}
\newtheorem{Prop}[Thm]{Proposition}
\newtheorem{Cor}[Thm]{Corollary}
\theoremstyle{definition}
\newtheorem{Ex}[Thm]{Example}
\newtheorem{Rem}[Thm]{Remark}

\DeclareMathOperator{\Aut}{Aut}
\DeclareMathOperator{\Autd}{Aut^{op}}
\DeclareMathOperator{\FrG}{FG}
\DeclareMathOperator{\FrR}{FR}
\DeclareMathOperator{\FrRS}{FRS}
\newcommand{\FR}{\FrR(\Omega)}
\newcommand{\FRS}{\FrRS(\Omega)}
\newcommand{\FG}{\FrG(\Omega)}
\newcommand{\op}[1]{#1^\mathrm{op}}

\renewcommand{\a}{\mathbf{a}}
\renewcommand{\c}{\mathbf{c}}
\renewcommand{\d}{\mathbf{d}}
\renewcommand{\u}{\mathbf{u}}
\newcommand{\s}{\mathbf{s}}
\renewcommand{\t}{\mathbf{t}}
\newcommand{\z}{\mathbf{z}}
\renewcommand{\r}{\mathbf{r}}
\newcommand{\x}{\mathbf{x}}
\newcommand{\y}{\mathbf{y}}
\renewcommand{\phi}{\varphi}
\newcommand{\C}{\mathcal{C}}

\newcommand{\X}{\mathcal{X}}
\newcommand{\Y}{\mathcal{Y}}
\newcommand{\1}{\mathbf{1}}
\newcommand{\we}{\wedge}

\newcommand{\da}{^\downarrow}
\newcommand{\ol}[1]{\overline{#1}}

\def\NN{{\mathbb N}}

\begin{document}

\title[Embedding of restriction semigroups]{Embedding in 
factorisable restriction
monoids}

\author{Victoria Gould}
\address{Department of Mathematics\\ University of York\\ Heslington\\ York\\ YO10 5DD\\ UK}
\email{victoria.gould@york.ac.uk}

\author{Mikl\'os Hartmann}
\address{Bolyai Institute, University of Szeged, Aradi v\'ertan\'uk tere 1, Szeged, Hungary, H-6720}
\email{hartm@math.u-szeged.hu}

\author{M\'aria B.\ Szendrei}
\address{Bolyai Institute, University of Szeged, Aradi v\'ertan\'uk tere 1, Szeged, Hungary, H-6720}
\email{m.szendrei@math.u-szeged.hu}

\date{January 18, 2016}

\thanks{Research partially supported by the Hungarian National Foundation for Scientific Research grants no.~K083219, K104251, PD115705, and by EPSRC grant no.\ EP/IO32312/1.
\vskip 3pt\noindent
{\it Mathematical Subject Classification (2010):} 20M10, 20M05.\\
{\it Key words:} Restriction semigroup, weakly $E$-ample semigroup, (almost) factorisability, semidirect product.}

\begin{abstract}
Each restriction semigroup is proved to be embeddable in a factorisable restriction monoid, or, equivalently,
in an almost factorisable restriction semigroup. 
It is also established that each restriction semigroup has a proper cover 
which 
is embeddable in a semidirect product of a semilattice by a group. 
\end{abstract}

\maketitle

\section{Introduction}

Restriction semigroups are non-regular generalisations of inverse semigroups.
They are semigroups equipped with two additional unary operations which satisfy certain identities. 
In particular, each inverse semigroup determines a restriction semigroup where the unary operations
assign the idempotents $aa^{-1}$ and $a^{-1}a$, respectively, to any element $a$.
The class of restriction semigroups is just the variety of algebras generated by these restriction semigroups
obtained from inverse semigroups, see \cite{free-amp}. Restriction semigroups  (formerly called weakly $E$-ample semigroups) have arisen from a number of mathematical perspectives. 
For a historical overview of restriction semigroups   and a detailed introduction to their fundamental properties
the reader is referred to \cite{http}.

So far, a number of important results of the rich structure theory of inverse semigroups have been recast in the broader setting of  restriction semigroups.
The current paper is a contribution to this body of work.

In the theory of inverse semigroups, semidirect products of semilattices by groups play an important role.
One of the reasons for this is that every inverse semigroup divides such a semidirect product
(see \cite{L=8} for the main results in this area and for references).
In fact, something stronger is true, namely that every inverse semigroup can be embedded into an (idempotent separating) 
homomorphic image of such a semidirect product.
One way to see this result is through extensions of partial isomorphisms of semilattices: an inverse semigroup determines partial isomorphisms of 
its semilattice, and constructing the corresponding semidirect product in fact includes embedding the semilattice into a bigger one, 
and finding a group acting on the bigger semilattice such that the original partial isomorphisms are restrictions of the automorphisms 
determined by the group action.

The situation for restriction semigroups is similar, but also different in some crucial respects. 
In this case groups are replaced by monoids, which gives rise to several problems.
The first result analogous to that formulated above for inverse semigroups was obtained by the third author in \cite{M2013}: 
she has shown that every restriction semigroup can be embedded into a (projection separating) homomorphic image of
a so-called $W$-product of a semilattice by a monoid, that is, into an almost left factorisable restriction semigroup (cf.\ \cite{GSz}). 
An alternative proof, based on the idea of extending partial isomorphisms to injective endomorphisms, has been presented
by Kudryavtseva in \cite{K2014}.

In general, the monoid acts in a $W$-product by injective endomorphisms on the semilattice, and this was the case in the results
of \cite{M2013} and \cite{K2014} mentioned. In particular, if the monoid acts by automorphisms on the semilattice then
the $W$-product becomes a semidirect product.

In this paper we obtain the strongest possible result in this direction by showing that the monoid and the semilattice can be chosen such that 
the monoid acts on the semilattice by automorphisms, 
thus demonstrating that any restriction semigroup can be embedded into a (projection separating) homomorphic image of 
a semidirect product of a semilattice by a monoid 
where the monoid acts on the semilattice by automorphisms.
In other words, we prove that  
any restriction semigroup can be embedded into an almost factorisable restriction semigroup,
or, equivalently, into a factorisable restriction monoid
(cf.\ \cite{GSz}, \cite{M2013}).
Furthermore, the construction applied in the proof of this result allows us to show that the proper cover embeddable in a $W$-product,
which is
provided for a restriction semigroup in \cite{M2012} and \cite{K2014}, is embeddable also in a semidirect product of a semilattice by a group.

Note that restriction semigroups in general, 
proper restriction semigroups, almost factorisable restriction semigroups
and factorisable restriction monoids 
are defined in a left-right dual manner, and
semidirect products of semilattices by monoids where monoids act on semilattices by automorphisms are left-right symmetric.
However,
this is not the case with almost left factorisable restriction semigroups or $W$-products. 
Therefore the embedding results of the present paper reflect the left-right symmetry of restriction and proper restriction semigroups, 
while this was not the case with the foregoing results in \cite{M2012} and \cite{K2014}.
In the last section, 
we strengthen some results in \cite{K2014} and \cite{Jones} by proving that each ultra $F$-restriction, or, 
equivalently, each perfect restriction  monoid $S$ 
whose greatest reduced factor is a free monoid is embeddable in a semidirect product of a semilattice by a monoid 
in such a way that 
the monoid acts on the semilattice by automorphisms, and
all congruences of $S$ extend to the semidirect product.

\section{Preliminaries}

In general, mappings and partial mappings are considered as right operands, and so their products are calculated from the left to the right.
This applies, amongst others, to the group of automorphisms $\Aut Y$ and to the Munn semigroup $T_Y$ of a semilattice $Y$, 
and to the symmetric inverse monoid $I(X)$ on a set $X$.
By writing $\Autd Y$, $\op{T}_Y$, etc., we indicate that the (partial) mappings in them are considered left operands, and their products are calculated
from the right to the left.

\subsection{Restriction semigroups} \cite{http}\quad
A {\em left restriction semigroup} is defined to be an algebra of type $(2,1)$, more precisely, an algebra $S=(S;\cdot,{}^+)$
where $(S;\cdot)$ is a semigroup and ${}^+$ is a unary operation such that the following
identities are satisfied:
\begin{equation}\label{id-lr}
x^+ x = x,\quad 
x^+ y^+=y^+ x^+,\quad
(x^+ y)^+= x^+ y^+,\quad 
xy^+=(xy)^+ x.
\end{equation}
A {\em right restriction semigroup} is defined dually, that is, it is an algebra $S=(S;\cdot,{}^*)$
satisfying the duals of the identities (\ref{id-lr}). Finally, if $S=(S;\cdot,{}^+,{}^*)$ is an algebra  of type $(2,1,1)$ where
$S=(S;\cdot,{}^+)$ is a left restriction semigroup, $S=(S;\cdot,{}^*)$ is a right restriction semigroup and
the identities
\begin{equation}\label{id-rp}
{(x^+)}^*=x^+,\quad
{(x^*)}^+=x^*
\end{equation}
hold then it is called a {\em restriction semigroup}.
Notice that the defining properties of a restriction semigroup are left-right dual. 
Therefore in the sequel dual definitions and statements will not be explicitly formulated.

Among restriction semigroups, the notions of a subalgebra, homomorphism, congruence and factor algebra are understood in type $(2,1,1)$,
which is emphasised by using the expressions 
$(2,1,1)$-subsemigroup, $(2,1,1)$-morphism, $(2,1,1)$-congruence and $(2,1,1)$-factor semigroup, respectively.
A restriction semigroup with identity element is also called a {\em restriction monoid}. 

In particular, each monoid $M$ becomes a restriction monoid
by defining $a^+=a^*=1$ for any $a\in M$. 
It is easy to see that these restriction semigroups  
are just those with both unary operations being constant.
Such a restriction semigroup (monoid) 
is called {\em reduced}.
Notice that the submonoids, congruences, etc.\ of monoids and 
the $(2,1,1)$-subsemigroups, $(2,1,1)$-congruences, etc.\ of the reduced
restriction semigroups (monoids) obtained from them coincide. 
Therefore we often consider reduced restriction semigroups (monoids) just as monoids, and vice versa.
Similarly, each semilattice $Y$ becomes a restriction semigroup
by defining $a^+=a^*=a$ for every $a\in Y$. 
These restriction semigroups  
are just those where both unary operations equal the identity map.
Clearly, the $(2,1,1)$-subsemigroups, $(2,1,1)$-congruences, etc.\ of such a
restriction semigroup $Y$ are just the subsemilattices, congruences, etc.\ of the semilattice $Y$.
Therefore such a restriction semigroup is simply considered and called a semilattice.

Let $S$ be any restriction semigroup. By (\ref{id-rp}), we have $\{x^+:x\in S\}=\{x^*:x\in S\}$. This set is called
{\em the set of projections of $S$}, and is denoted by $P(S)$.
By (\ref{id-lr})
and its dual, $P(S)$ 
can be seen to be
a $(2,1,1)$-subsemigroup in $S$
which is a semilattice. 
Notice that a restriction semigroup $S$ is reduced if and only if $P(S)$ is a singleton, and, if this is the case, then
the unique element of $P(S)$ is the identity element of $S$.

Given a restriction semigroup $S$, we define a relation $\le$ on $S$ such that, for every $a,b\in S$, 
$$a\le b\quad \hbox{if and only\ if}\quad a=a^+b.$$
It is easy to see that $a \le b$ if and only if $a=eb$ for some $e \in P(S)$.
Observe that the dual of this relation is the same since $a=a^+ b$ implies $a=b(a^+b)^*=ba^*$,
and the dual implication is also valid. 
The relation $\le$ is a compatible partial order on $S$, and it extends the natural partial order of
the semilattice $P(S)$. It is called the {\em natural partial order on $S$}.

We also consider a relation on $S$, denoted by $\sigma_S$, or simply $\sigma$: for any $a,b\in S$, let 
$$a\ \sigma\ b\quad \hbox{if and  only  if }\quad ea=eb \quad {\rm for\ some}\quad e\in P(S).$$
Again notice that if there exists $e\in P(S)$ with $ea=eb$ then there exists also $f\in P(S)$ with $af=bf$, and conversely.
Therefore the relation defined dually to $\sigma$ coincides with $\sigma$. 
The relation $\sigma$ is the least congruence on $S=(S;\cdot)$ where $P(S)$ is in a congruence class, which we denote by $P(S)\sigma$. 
Consequently, it is
the least $(2,1,1)$-congruence $\rho$ on $S=(S;\cdot,{}^+,{}^*)$ such that the $(2,1,1)$-factor semigroup $S/\rho$ is reduced.
Therefore we call $\sigma$ {\em the least reduced $(2,1,1)$-congruence on $S$\/}. Obviously, $P(S)\sigma$ is
the identity element of $S/\sigma$. 
The reduced restriction monoid $S/\sigma$ is often considered just as a monoid $S/\sigma=(S/\sigma;\cdot,P(S)\sigma)$.

Let $S,T$ be restriction semigroups. 
We say that $S$ is {\em proper} if, for every $a,b\in S$, the relations $a^+=b^+$ and $a\,\sigma\,b$ together imply $a=b$ and the dual property also holds.
It is easy to see that 
each $(2,1,1)$-subsemigroup of a proper restriction semigroup is proper.
A $(2,1,1)$-morphism $\varphi\colon T\to S$ is called {\em projection separating} if $e\varphi=f\varphi$ implies $e=f$ for every $e,f\in P(T)$.
We say that $T$ is a {\em proper cover} of $S$ if $T$ is proper and there exists a projection separating $(2,1,1)$-morphism from $T$ onto $S$.

\subsection{Semidirect products and factorisability}\label{sd-f} \cite{GSz}, 
 \cite{M2013}\quad
Let $T$ be a monoid and $Y=(Y;\we)$ a semilattice. 
We say that {\em $T$ acts on $Y$ (on the left) by automorphisms} if a monoid homomorphism $\alpha\colon T\to \Autd Y,\ t\mapsto\alpha_t$ is given. 
The element $\alpha_t a\ (t\in T,\ a\in Y)$ is usually denoted by $t\cdot a$.
The fact that $T$ acts on $Y$ by automorphisms is equivalent to requiring that 
the mapping $\alpha_t\colon Y\to Y,\ a\mapsto t\cdot a$ is bijective for every $t\in T$, and
the following equalities are valid for every $a,b\in Y$ and $t,u\in T$: 
\begin{equation}\label{act0}
t\cdot (a\we b) = t\cdot a\we t\cdot b,\quad u\cdot(t\cdot a) = (ut)\cdot a,\quad 1\cdot a=a.
\end{equation}
The {\em semidirect product} $Y\rtimes T$ is the algebra of type $(2,1,1)$ defined on the set 
$Y\times T=\{(t\cdot a,t):a\in Y,\ t\in T\}$ 
with the following operations:
\[
(t\cdot a,t)(u\cdot b,u)=(t\cdot(a\we u\cdot b),tu),\quad
(t\cdot a,t)^+=(t\cdot a,1)\quad\hbox{and}\quad (t\cdot a,t)^*=(a,1).
\]
Defining the multiplication this way may seem awkward, however, note that since $T$ acts by automorphisms, the multiplication and the $^+$ operation are in fact the same as in the `usual' semidirect product.
The reason for this unusual formulation is in the definition of the operation $^*$: since $T$ is not necessarily a group, one cannot use $t^{-1}$.
It is routine to check that $Y\rtimes T$ is a proper restriction semigroup where
$P(Y\rtimes T)=\{(a,1):a\in Y\}$ is isomorphic to $Y$, the relation $\sigma$ is just the kernel of the
second projection $Y\rtimes T\to T,\ (t\cdot a,t)\mapsto t$, and so $(Y\rtimes T)/\sigma$ is isomorphic to $T$.
Moreover, the natural partial order is the following relation:
$(t\cdot a,t)\le(u\cdot b,u)$ if and only if $a\le b$ and $t=u$.
The semidirect product $Y\rtimes T$ is a (proper restriction) monoid if and only if $Y$ has an identity element.
In this case the fact that $T$ acts by automorphisms implies that $t\cdot {\bf 1}={\bf 1}$ for any $t \in T$ where ${\bf 1}$ is the identity element of $Y$.
To avoid confusion, we call the attention to the fact that, from now on, when speaking about a semidirect product of a semilattice by a monoid, it is always
meant to be a restriction semigroup just defined. In particular, the action of the monoid on the semilattice is taken  to be an action by automorphisms.

Note that the semidirect product $Y\rtimes T$ is isomorphic to the reverse semidirect product $T\ltimes Y$
where the right action of $T$ on $Y$ is given by the monoid homomorphism $T\to\Aut Y,\ t\mapsto \alpha_t^{-1}$.
By definition (see \cite{GSz}), the reverse semidirect product $T\ltimes Y$ is, in fact, a $W$-product $W(T,Y)$.
Therefore the same is the case with the semidirect product $Y\rtimes T$.

In this paper, an action of a monoid $T$ on a semilattice $Y$ with automorphisms will frequently come as a restriction of an action of a group $G$
to its submonoid $T$.
In this case, we will prefer representing the elements of the semidirect product $Y\rtimes T$, as it is usual within the inverse semigroup
$Y\rtimes G$, in the form $(a,t)$ with $a\in Y,\ t\in T$.
In this form, the operations are:
\[
(a,t)(b,u)=(a\we t\cdot b,tu),\quad
(a,t)^+=(a,1)\quad\hbox{and}\quad (a,t)^*=(t^{-1}\cdot a,1),
\]
and the natural partial order is $(a,t)\le(b,u)$ if and only if $a\le b$ and $t=u$.

Note that, instead of assuming that a monoid $T$ is embeddable in a group and it acts on a semilattice $Y$ by automorphisms, 
we can require without loss of generality that $T$ is a submonoid in a group $G$ which it generates, $G$ acts on $Y$, 
and the action of $T$ is the restriction of the action of $G$.
For, recall first that if $T$ is embeddable in a group then the inclusion map $\kappa\colon T\to G_T$ where
$G_T=\langle T \!\mid\! \Xi\rangle$ is given by the defining relations $\Xi=\{tuv^{-1}: t,u,v\in T\ \hbox{and}\ tu=v\ \hbox{in}\ T\}$
is an embedding.
Second, if the action of $T$ on $Y$ is defined by the homomorphism $\alpha\colon T\to \Autd Y$ then
$\alpha_t\alpha_u\alpha_v^{-1}$ is the identity automorphism for every $t,u,v\in T$ with $tu=v$, and so
there exists a homomorphism $\beta\colon G_T\to \Autd Y$ such that $\alpha=\kappa\beta$.
Thus $\beta$ defines an action of $G_T$ on $Y$, and its restriction to $T$ is just $\alpha$.

Factorisable restriction monoids and almost factorisable restriction semigroups are introduced in \cite{M2013} (see also \cite{GSz}) analogously to 
factorisable inverse monoids and almost factorisable inverse semigroups, respectively.
A restriction monoid $F$ is called {\em factorisable} if $F=P(F)U(F)$ where $U(F)=\{u\in F : u^+=u^*=1\}$ is the  reduced $(2,1,1)$-subsemigroup of $F$ analogous to the group of units in an inverse monoid.
Note that $F=P(F)U(F)$ if and only if $F=U(F)P(F)$, therefore factorisability is a left-right symmetric property.
The notion of an {\em almost factorisable} restriction semigroup is defined by means of permissible subsets but, instead of the definition,
it suffices for the purposes of this paper to recall how they relate to factorisable restriction monoids (\cite[Proposition 3.10 and Theorem 3.11]{M2013}).
If $F$ is a factorisable restriction monoid then both $F$ and the $(2,1,1)$-subsemigroup $F\setminus U(F)$ of $F$ are almost factorisable, and conversely, 
each almost factorisable restriction semigroup is, up to $(2,1,1)$-ismomorphism, of the latter form.
In particular, this implies that a restriction semigroup is $(2,1,1)$-embeddable in a factorisable restriction monoid if and only if it is $(2,1,1)$-embeddable in an
almost factorisable restriction semigroup.

In the proof of the main result of the paper, we need the following description 
of these restriction semigroups by means of semidirect products.

\begin{Prop} \cite[Theorem 3.12]{M2013} 
A restriction monoid (semigroup) is factorisable (almost factorisable) if and only if it is
a $(2,1,1)$-morphic image of a semidirect product of a semilattice with identity (a semilattice) by a monoid.
\end{Prop}

Given an almost factorisable restriction semigroup $F^-$ in the above form, now we establish for later use, how to construct a factorisable restriction monoid 
$F$ such that $F^-=F\setminus U(F)$.

\begin{Lem}\label{new}
Consider an almost factorisable restriction semigroup $F^-=(Y\rtimes T)/\rho$, where $T$ is a monoid acting on the semilattice $Y$ by automorphisms, and
$\rho$ is a $(2,1,1)$-congruence on the restriction semigroup $Y\rtimes T$.
Let $Y^e$ be the semilattice obtained from $Y$ by adjoining an identity element $e\notin Y$ even if $Y$ has an identity element.
Extend the action of $T$ on $Y$ to $Y^e$ by putting $t\cdot e=e$ for every $t\in T$, and consider the relation
$\rho_e=\rho\cup\imath$
on the semidirect product $Y^e\rtimes T$ where $\imath$ is the equality relation on $U=(Y^e\rtimes T)\setminus (Y\rtimes T)$.
Then $\rho_e$ is a $(2,1,1)$-congruence on the restriction monoid $Y^e\rtimes T$ extending $\rho$, $F=(Y^e\rtimes T)/\rho_e$ is a factorisable restriction monoid
with $U(F)=\{u\rho_e: u\in U\}=\{\{(e,t)\}:t\in T\}$, 
and the mapping $F^-\to F,\ (a,t)\rho\mapsto (a,t)\rho_e\ ((a,t)\in Y\rtimes T)$ is an injective $(2,1,1)$-morphism with range $F\setminus U(F)$.
\end{Lem}

\begin{proof}
Obviously, $T$ acts on $Y^e$ by automorphisms, the semidirect product $Y^e\rtimes T$ is a restriction monoid with $U(Y^e\rtimes T)=U$, and
$Y\rtimes T$ is a $(2,1,1)$-subsemigroup and ideal of  $Y^e\rtimes T$. 
Moreover, it is also clear that the relation $\rho_e$ is an equivalence whose restriction to $Y\rtimes T$ is $\rho$, and that
it is compatible with both unary operations.
It remains to check that $\rho_e$ is compatible with the multiplication.
We have to verify that if $(a,t)\,\rho\,(b,u)$ in $Y\rtimes T$ and $v\in T$ then
both $(a,t)(e,v)\,\rho\,(b,u)(e,v)$ and $(e,v)(a,t)\,\rho\,(e,v)(b,u)$ hold,
that is, both $(a,tv)\,\rho\,(b,uv)$ and $(v\cdot a,vt)\,\rho\,(v\cdot b,vu)$ are valid.
We verify the first relation, the second being similar.
Since $\rho$ is a $(2,1,1)$-congruence on $Y\rtimes T$, we have 
$(a\wedge t\cdot c,tv)=(a,t)(c,v)\,\rho\,(b,u)(c,v)=(b\wedge u\cdot c,uv)$ 
for every $c\in Y$.
Choosing $c$ to be the unique element of $Y$ such that $t\cdot c=a$, we obtain that 
$(a,tv)\rho=(b\wedge u\cdot c,uv)\rho\le (b,uv)\rho$ in $F^-$.
Changing the roles of $(a,t)$ and $(b,u)$ the equality follows, proving the first relation.
\end{proof}

\subsection{Congruences on restriction semigroups generated by relations} \cite{M2013}\quad
Given a set $X$ of variables, by a {\em (restriction) term in $X$} we mean a formal expression built up from the elements of $X$ by means of
the operational symbols --- the binary operational symbol $\cdot$ and the unary operational symbols ${}^+$ and ${}^*$ --- in finitely
many steps. Since the binary operation $\cdot$ is always interpreted in a semigroup we delete the unnecessary parentheses from the terms. 
If $S$ is a fixed restriction semigroup then 
we introduce a nullary operational symbol for every element $s$ in $S$, and, for simplicity, denote it also by $s$. 
By a {\em polynomial of $S$} we mean an expression obtained in a way similar to terms, 
but from variables and these nullary operational symbols.
A polynomial can also be interpreted in the way that such nullary operational symbols --- briefly elements of $S$ --- are substituted for certain
variables in a term.
So a {\em unary polynomial of $S$}, that is, a polynomial in one variable $x$, is of the form $\t(x,s_1,\ldots,s_{k-1})$ for some
term $\t$ in $k$ variables and for some elements $s_1,\ldots,s_{k-1}\in S$.
To simplify our notation, we denote by $S^\star$ the set of finite sequences of elements of $S$, and we use Greek letters to denote elements of $S^\star$.
The length of the sequences will be determined by the context where they appear.
For example, if $\t$ is a term in $k$ variables and 
we are talking about $\t(x,\alpha)$ for some $\alpha\in S^\star$ then the length of $\alpha$ is $k-1$.

Let $S$ be a restriction semigroup and let $\tau \subseteq S \times S$ be a symmetric relation.
We denote by $\tau^\#$ the $(2,1,1)$-congruence on $S$ generated by $\tau$.
It is a basic fact of universal algebra that if $s,t \in S$ then $s \ \tau^\#\ t$ if and only if there exist a sequence $p_1(x),p_2(x),\ldots, p_k(x)$ of unary polynomials of $S$ and 
elements $c_1,d_1,\ldots, c_k,d_k \in S$ such that $(c_i,d_i) \in \tau$ for all $i\ (1\leq i\leq k)$ and
\[
s=p_1(c_1), p_1(d_1)=p_2(c_2),\ldots, p_k(d_k)=t.
\]
In general, unary polynomials over a restriction semigroup can be quite complicated.
However, it is shown in \cite{M2013} that the unary polynomials appearing in the sequence can be chosen to be simpler, and this will be key to our later arguments.

We define two sequences of terms in variables $x,y,z,y_0,z_0,\ldots$ in the following way: let
\begin{equation*}
\t _+^{(0)}(x,y_0,z_0)=(y_0xz_0)^+,\quad
\t _*^{(0)}(x,y_0,z_0)=(y_0xz_0)^*,
\end{equation*}
and, for every $i\in\NN$, let
\begin{eqnarray*}
\t _+^{(i)}(x,y_0,z_0,\ldots,y_{i-2},z_{i-1},y_i)&\!\!\!=\!\!\!&\big(y_i\t _*^{(i-1)}(x,y_0,z_0,\ldots,y_{i-2},z_{i-1})\big)^+,\\
\t _*^{(i)}(x,y_0,z_0,\ldots,z_{i-2},y_{i-1},z_i)&\!\!\!=\!\!\!&\big(\t _+^{(i-1)}(x,y_0,z_0,\ldots,z_{i-2},y_{i-1})z_i\big)^*.
\end{eqnarray*}
For convenience we note that
\[
\begin{array}{c}
\t_+^{(1)}(x,y_0,z_0,y_1) = \big(y_1(y_0xz_0)^*\big)^+, \\
\t_*^{(1)}(x,y_0,z_0,z_1) = \big((y_0xz_0)^+ z_1\big)^*, \\
\t_+^{(2)}(x,y_0,z_0,z_1,y_2) = \big(y_2\big((y_0xz_0)^+ z_1\big)^*\big)^+, \\
\t_*^{(2)}(x,y_0,z_0,y_1,z_2) = \big(\big(y_1(y_0xz_0)^*\big)^+z_2\big)^*.
\end{array}
\]
We define the following sets of terms:
\[
\mathbf{T}_{(+)}=\{\t^{(i)}_+,\t^{(i+1)}_*:i\in \mathbb{N}_0 \text{ is even}\},\ 
\mathbf{T}_{(*)}=\{\t^{(i)}_*,\t^{(i+1)}_+: i \in \mathbb{N}_0 \text{ is even}\},
\]
\[
\mathbf{T}=\{ \t: \t(x,y,z)=yxz \text{ or } \t(x,y,z,y_0,z_0,\ldots)=y\u z \text{ where } \u\in \mathbf{T}_{(+)} \cup \mathbf{T}_{(*)}\}.
\]
Notice that $\mathbf{T}_{(+)}$ ($\mathbf{T}_{(*)}$) contains terms with $+$ ($*$) as innermost unary operation.
The set $\mathbf{T}_{(+)}$ ($\mathbf{T}_{(*)}$) differs from  the set $\mathbf{T}_+$ ($\mathbf{T}_*$) used in \cite{M2013}
since the latter consists of the terms with $+$ ($*$) as outermost unary operation. 
The reason for this modification (and slightly awkward notation) is that, in contrast to the proof of the main result of \cite{M2013}, 
we are going to distinguish cases depending on the innermost unary operation. 

It is shown in \cite{M2013} that the terms in $\mathbf{T}$ are enough to determine a $(2,1,1)$-congruence generated by a symmetric relation.

\begin{Prop} \cite[Proposition 4.5]{M2013}
If $\tau$ is a symmetric relation on a restriction semigroup $S$ then, for any $s,t \in S$, we have $s\ \tau^\#\ t$ if and only if $s=t$ or 
there exist $k \in \mathbb{N}$, $\t_1,\ldots, \t_k \in \mathbf{T}$, $\alpha_1,\ldots,\alpha_k \in S^*$ and $(c_1,d_1), \ldots, (c_k,d_k) \in \tau$ such that
\[
s=\t_1(c_1,\alpha_1), \t_1(d_1,\alpha_1)=\t_2(c_2,\alpha_2), \ldots, \t_k(d_k,\alpha_k)=t.
\]
\end{Prop}

Let $S,S'$ be restriction semigroups such that $S$ is a $(2,1,1)$-subsemigroup of $S'$, and let $\rho$ be
$(2,1,1)$-congruence on $S$.
A $(2,1,1)$-congruence $\rho'$ on $S'$ is said to {\em extend} $\rho$ if the restriction of $\rho'$ to $S$ equals $\rho$.
We say that $\rho$ {\em extends to $S'$} if there exists a $(2,1,1)$-congruence $\rho'$ on $S'$ extending $\rho$.
Notice that $\rho$ extends to $S'$ if and only if the $(2,1,1)$-congruence on $S'$ generated by $\rho$ extends $\rho$.

\subsection{Free restriction monoids}\label{free-restr} \cite{free-amp}\quad
A transparent model of the free restriction monoid on a set is given in \cite{free-amp} 
as a full subsemigroup in the free inverse monoid on the same set  (cf.\ \cite{L=8}). 

Let $\Omega$ be a non-empty set, and denote by $\Omega^*$ and $\FG$, respectively, 
the usual models of the free monoid and the free group on $\Omega$ where $\Omega^*$ consists of all words 
(i.e., of finite sequences) over $\Omega$ and $\FG$ consists of all reduced words over $\Omega\cup\Omega^{-1}$.
The empty word is denoted by $1$.
Obviously, $\Omega^*$ is a subsemigroup of $\FG$, and $1$ is the identity element in both. 

Let $\C$ be the Cayley graph of $\FG$ which is well known to be a tree.
Therefore each finite subset $A \subseteq \FG$ determines a maximum subgraph of $\C$ having $A$ as its set of vertices.
From now on, we identify $A$ with this subgraph.
Let $\X$ be the set of all finite connected subgraphs of $\C$ (i.e., 
of all finite subtrees), partially ordered by reverse inclusion.
Then $\X$ is a semilattice without an identity.
Clearly, $\Y=\{A\in\X : 1\in A\}$ is a principal order ideal of $\X$ and a subsemilattice with identity element $\1=\{1\}$.
Moreover, $(\FG,\X,\Y)$ is a McAlister triple where $\FG$ acts on the partially ordered set $\X$ by multiplication:
if $g \in \FG$ and $A \in \X$ then $g\cdot A=\{ga: a \in A\}$.
It is well known that the $P$-semigroup $P(\FG,\X,\Y)$ is a free inverse monoid on $\Omega$.

It is established in \cite{free-amp} that 
\[
\FR=\{(A,\ol{a})\in P(\FG,\X,\Y): \ol{a} \in \Omega^*\}=\{(A,\ol{a}) \in \Y \times \Omega^*: A \supseteq \ol{a} \cdot \1\}
\]
is a $(2,1,1)$-subsemigroup in $P(\FG,\X,\Y)$, it is a free restriction monoid on $\Omega$
and $\FRS=\FR\setminus\{(\1,1)\}$ is a free restriction semigroup on $\Omega$.

Since the free inverse monoid $P(\FG,\X,\Y)$ is $F$-inverse, the partially ordered set $\X$, where the partial order will subsequently be denoted $\le$, 
is a semilattice.
Thus $\Y$ is a principal ideal in it with identity element $\1$, 
and the action of $\FG$ on the partially ordered set $\X$ is, in fact, an action on the semilattice $\X$.
Since $\C$ is a tree, for every finite subgraph $A \subseteq \FG$, there exists a minimum subtree $A'$ containing $A$.
The semilattice operation $\we$ of $\X$ can be given by the rule 
$ A\we B=(A \cup B)'\ (A,B \in \X)$.
Thus the semidirect products $\X\rtimes\FG$ and $\X\rtimes\Omega^*$ are defined, and $\FR$ can be given in this context as
\[
\FR=\{(A,\ol{a})\in \X \rtimes \Omega^*: A \in \Y,\ \ol{a} \in A\}=\{(A,\ol{a}) \in \Y \times \Omega^*: A \leq \ol{a} \cdot \1\}.
\]

Note that the free inverse monoid $P(\FG,\X,\Y)$ is usually considered as an inverse subsemigroup of a semidirect product of
another semilattice by $\FG$, namely, of the semilattice of all finite subgraphs of $\C$ with respect to the usual join.
In the sequel, we need the approach in the previous paragraph, as it was the case with the $W$-product applied in \cite{M2012} and \cite{M2013}.

\section{Main result}

Our aim is to show that every restriction monoid, and consequently, every restriction semigroup is $(2,1,1)$-embeddable in a factorisable restriction monoid, or equivalently in an almost factorisable restriction semigroup.
We obtain a $(2,1,1)$-embedding by taking a free restriction monoid $\FR$, having a given restriction monoid $S$ as its image via a $(2,1,1)$-morphism, say, $\varphi$, 
and by considering the $(2,1,1)$-embedding of
$\FR$ into the semidirect product $\X \rtimes \Omega^*$ 
described in the previous section where the monoid $\Omega^*$ acts on the semilattice $\X$ by automorphisms.
The kernel of $\varphi\colon \FR\to S$ is a $(2,1,1)$-congruence on $\FR$, we consider the $(2,1,1)$-congruence on the semidirect product $\X \rtimes \Omega^*$ it generates.
If the restriction of the latter congruence to $\FR$ equals the kernel, then we obtain 
a $(2,1,1)$-embedding of $S$ into a $(2,1,1)$-morphic image of the semidirect product, 
that is, into an almost factorisable restriction semigroup.
As it turns out, the fact that this approach works depends only on the fact that the actions of the free monoid $\Omega^*$ and 
of the free group $\FG$ on the semilattice $\X$ satisfy a simple condition.
This allows us to slightly generalise our approach.

From now on, let $\X$ be a semilattice and let $\Y$ be a principal ideal of $\X$.
Denote the identity element of $\Y$ by $\1$.
Let $T$ be a monoid embeddable in a group, and suppose that $T$ acts on $\X$ (on the left) by automorphisms.
We have seen formerly that in this case, there exists a group $G$ acting on $\X$ such that $T$ is a submonoid of $G$ generating $G$, 
and the action of $T$ is the restriction of the action of $G$.
Let us choose and fix such a group $G$.
In the sequel we will use boldface letters to represent the elements of the semidirect products $\X\rtimes G$ and $\X\rtimes T$, 
the corresponding capital letters and overlined letters to denote their first and second components, respectively 
--- for example, $\a=(A,\ol{a})$.

Define the following subset of the semidirect product $\X \rtimes T$:
\begin{equation}\label{R}
R=\{(A,\ol{a}) \in \Y \times T: A \leq \ol{a}\cdot \1\}.
\end{equation}
Note that if $(A,\ol{a})\in R$ and $A=\ol{a}\cdot B\le \ol{a}\cdot \1$ where $B\in\X$, then
as $T$ acts by automorphisms we have $B\le \1$.
It is now easy to check that $R$ is a $(2,1,1)$-subsemigroup in $\X \rtimes T$ with identity element $(\1,1)$, and so it is a proper restriction monoid.
Notice that, for any $(A,\ol{a}) \in \X \rtimes T$ there exists a maximum element of $R$ less than or equal to $(A,\ol{a})$ which we denote by $(A,\ol{a})\da$.
It is clear that $(A,\ol{a})\da=(\1\we A\we \ol{a}\cdot \1,\ol{a})$.

Observe that $G\cdot \Y$ is the smallest subsemilattice in $\X$ 
containing $\Y$ which is invariant under the action of $G$.
Since, for any $g,h\in G$ and $A,B\in\Y$, we have 
$g\cdot A\wedge h\cdot B=g\cdot(A\wedge g^{-1}h\cdot B)$, and $\Y$ is an ideal in $\X$,
we see that $G\cdot\Y$ is, indeed, a subsemilattice of $\X$.
Since we have $R\subseteq \Y\times T$, we can assume without loss of generality that $\X=G\cdot \Y$.

A sequence $w_1,w_2,\ldots,w_n$ of elements of $G$ is called {\em an alternating sequence in $T$} if, for every $i\ (1\leq i\leq n)$, 
we have $w_i=t_i^{\epsilon_i}$ for some $t_i\in T,\ \epsilon_i\in\{1,-1\}$, and, for every $i,j\ (1\leq i,j\leq n)$, 
we have $\epsilon_i=\epsilon_j$ if and only if $2\mid i-j$.
A factorisation $g=w_1w_2\cdots w_n$ of an element $g\in G$ is 
called {\em a nice factorisation in $T$ with respect to $\1$} if $w_1,w_2,\ldots,w_n$ is 
an alternating sequence in $T$ such that, for every $i\ (1\leq i < n)$, we have that
\[
w_i\cdot \1 \geq \1 \we w_i\cdots w_n\cdot \1.
\]
We say that the action of the monoid $T$ on the semilattice $\X$ (necessarily by automorphisms) is {\em nice 
over the group $G$ with respect to the element $\1\in\X$} 
if every element $g\in G$ has a nice
factorisation in $T$ with respect to $\1$.
Note that, in particular, the action of a group on a semilattice is always nice 
over the group itself with respect to any element of the semilattice.
We say that $\X\rtimes T$ is a 
{\em nice
semidirect product of a semilattice $\X$ by a monoid $T$} 
if $T$ is embeddable in a group and acts on $\X$ by automorphisms, $G$ is a group having $T$ as a submonoid and acting on $\X$ such that
the action of $T$ is the restriction of the action of $G$, and is nice
over $G$ with respect to an element $\1\in\X$ such that $\X=G\cdot \Y$ 
for the principal ideal $\Y$ of $\X$ with identity $\1$.

Let $G$ be a group acting on a semilattice $\X$, and let $\varepsilon$ be a congruence on $\X$.
We say that $\varepsilon$ is {\em $G$-invariant} if $x\;\varepsilon\;x'$ implies $g\cdot x\;\varepsilon\;g\cdot x'$
for every $x,x'\in\X$ and $g\in G$.
In this case, $G$ acts by automorphisms
on $\X/\varepsilon$ by the rule $g\cdot(x\varepsilon)=(g\cdot x)\varepsilon$.
The following observation is straightforward.

\begin{Lem}\label{per-eps}
If an action of a monoid $T$ on a semilattice $\X$ is nice
over a group $G$ with respect to an element $\1\in\X$
and $\varepsilon$ is a $G$-invariant congruence on $\X$
then the restriction of the action of $G$ on $\X/\varepsilon$ to $T$ is also nice 
over $G$ with respect to $\1\varepsilon\in\X/\varepsilon$.
Consequently, if $\X\rtimes T$ is a nice 
semidirect product then the same holds for $(\X/\varepsilon)\rtimes T$.
\end{Lem}

\begin{Ex}\label{Ex1}
The action of $\Omega^*$ on the semilattice $\X$ defined in Subsection \ref{free-restr} is 
 nice  
over $\FG$ with respect to $\1$.
It is easy to see that if $g \in \FG$, then the unique reduced alternating factorisation $g=w_1w_2\cdots w_n$ where
$w_1,w_2,\ldots,w_n\in\Omega^*\cup(\Omega^*)^{-1}$ is a 
 nice  
factorisation of $g$ in $\Omega^*$ with respect to $\1$.
Moreover, we have $\X=\FG\cdot \Y$ where $\Y$ is a principal ideal with identity $\1$, and so
$\X\rtimes \Omega^*$ is a 
 nice  
semidirect product.
\end{Ex}

\begin{Ex}
Let $G$ and $T$ be the free Abelian group and free commutative semigroup, respectively, on the set $\Omega$.
Then for every $g \in G$, there exist unique elements $u,t \in T$ such that $g=u^{-1}t$.
We say that a subset $A \subseteq G$ is {\em min-closed} if, for every $g=\omega_1^{a_1} \cdots \omega_s^{a_s},\ h=\omega_1^{b_1} \cdots \omega_s^{b_s} \in A$,
where $\omega_1,\ldots,\omega_s$ are pairwise distinct elements of $\Omega$,
we have that $\omega_1^{\text{min}(a_1,b_1)}\cdots \omega_s^{\text{min}(a_s,b_s)} \in A$.
Let $\X=\{A \subseteq G: A \neq \emptyset,\ A \text{ is min-closed}\}$ and, for all $A,B \in \X$, let $A \we B$ be the smallest min-closed subset of $G$ 
containing $A \cup B$.
Then $\X$ is a semilattice and $G$ acts on $\X$ by multiplication.
Furthermore, $\Y=\{A \in \X: 1 \in A\}$ is a principal ideal of $\X$ with identity $\{1\}=\1$.
Note that, for any $g \in G$, the unique factorisation $g=u^{-1}t$ is 
 nice  
in $T$ with respect to $\1$, because if $g=\omega_1^{a_1}\cdots \omega_s^{a_s}$ then 
$u^{-1}=\omega_1^{\text{min}(a_1,0)} \cdots \omega_s^{\text{min}(a_s,0)} \in \{1\} \we \{g\}$, showing that $u^{-1} \cdot \1 \geq \1 \we g \cdot \1$.
\end{Ex}

From now on, additionally to our former assumptions on $\X$, $\Y$, $T$ and $G$, we suppose that the action of $T$ on $\X$ is 
 nice  
over the group $G$ with respect to $\1$.

\begin{Lem} \label{Onedir}
Let $\t^{(i)} \in \mathbf{T}_{(+)} \cup \mathbf{T}_{(*)}$ and let $\alpha \in (\X \rtimes T)^\star$.
Then there exist $U,V \in \X$, $g \in G$ depending on $\alpha$ such that
\[
\t^{(i)}(\c,\alpha)=(U \we g\cdot \tilde{C} \we g\tilde{c} \cdot V,1)
\]
for every $\c \in R$ where
\[
\tilde{\c}=(\tilde{C},\tilde{c})=\left\{ \begin{array}{ll}
(C,\ol{c}) & \text{if } \t^{(i)} \in \mathbf{T}_{(+)}, \\
(c^{-1}\cdot C,\ol{c}^{-1}) & \text{if } \t^{(i)} \in\mathbf{T}_{(*)}.
\end{array}\right.
\]
\end{Lem}

\begin{proof}
We proceed by induction on $i$. 
If $i=0$ then $\t^{(i)}=(\y_0\x\z_0)^+$ or $\t^{(i)}=(\y_0\x\z_0)^*$, and
$\alpha=\big( (A,\ol{a}),(B,\ol{b})\big)$ for some $(A,\ol{a}),(B,\ol{b}) \in \X \rtimes T$.
If $\t^{(i)} \in \mathbf{T}_{(+)}$ then
\[
\t^{(i)}(\c,\alpha)=\big( (A,\ol{a})(C,\ol{c})(B,\ol{b})\big)^+=(A\ \we\ \ol{a}\cdot C\ \we\ \ol{a}\, \ol{c}\cdot B,1), 
\]
so we can set $U=A,\, V=B$ and $g=\ol{a}$ .
Similarly, if $\t^{(i)} \in \mathbf{T}_{(*)}$ then
\[
\t^{(i)}(\c,\alpha)=\big( (A,\ol{a})(C,\ol{c})(B,\ol{b})\big)^*=(\ol{b}^{-1}\ol{c}^{-1}\ol{a}^{-1} \cdot A \we \ol{b}^{-1}\ol{c}^{-1}\cdot C\we \ol{b}^{-1} \cdot B,1),
\]
so we can set $U=\ol{b}^{-1}\cdot B,\, V=\ol{a}^{-1}\cdot A$ and $g=\ol{b}^{-1}$.

Let $i\geq 1$ and $\t^{(i)} \in \mathbf{T}_{(+)} \cup \mathbf{T}_{(*)}$, and suppose that we have proven the statement for $i-1$.
Then $\t^{(i)}=\big( y_i \t^{(i-1)} \big)^+$ or $\t^{(i)}=\big( \t^{(i-1)} z_i \big)^*$ for an appropriate $\t^{(i-1)} \in \mathbf{T}_{(+)} \cup \mathbf{T}_{(*)}$ and 
$\alpha=\big(\alpha', (A,\ol{a})\big)$ where $\alpha' \in (\X \rtimes T)^\star$ and $(A,\ol{a}) \in \X \rtimes T$.
Note that $\t^{(i)} \in \mathbf{T}_{(+)}$ if and only if $\t^{(i-1)} \in \mathbf{T}_{(+)}$.
For, $\mathbf{T}_{(+)}$ contains terms with $+$ as innermost unary operation.
By the induction hypothesis there exist $U',V' \in \X$ and $g' \in G$ such that $\t^{(i-1)}(\c,\alpha')=(U' \we g'\cdot \tilde{C} \we g'\tilde{c} \cdot V',1)$ for all $\c \in R$.
If $\t^{(i)}=\big( y_i \t^{(i-1)}\big)^+$ then
\[
\t^{(i)}(\c,\alpha)=\big( (A,\ol{a}) (U'\we g'\cdot \tilde{C}\we g'\tilde{c} \cdot V',1)\big)^+=(A\we \ol{a} \cdot U' \we \ol{a}g'\cdot \tilde{C} \we \ol{a}g'\tilde{c}\cdot V',1),
\]
so we can set $U=A \we \ol{a}\cdot U', V=V'$ and $g=\ol{a}g'$.
Similarly, if $\t^{(i)}=\big( \t^{(i-1)} z_i\big)^*$ then
\[
\t^{(i)}(\c,\alpha)
=\big( (U' \we g' \cdot \tilde{C}\we g'\tilde{c} \cdot V',1)(A,\ol{a})\big)^*
=(\ol{a}^{-1}\cdot U'\we \ol{a}^{-1}g'\cdot \tilde{C}\we \ol{a}^{-1}g'\tilde{c} \cdot V' \we \ol{a}^{-1}\cdot A,1),
\]
so we can set $U=\ol{a}^{-1}\cdot U' \we \ol{a}^{-1}\cdot A, V=V'$ and $g=\ol{a}^{-1}g'$.
\end{proof}

The previous lemma showed that for all terms $\t \in \mathbf{T}_{(+)} \cup \mathbf{T}_{(*)}$, the elements $\t(\c,\alpha)$ ($\alpha \in (\X  \rtimes T)^\star$) are of special form.
Our aim is to replace in certain circumstances the sequence of constants $\alpha$ by $\beta \in R^\star$ and the term $\t$ by a term $\t'$ such that $\t(\c,\alpha)\da=\t'(\c,\beta)$ for all $\c \in R$.
The following lemma constitutes the first step in this direction.

\begin{Lem}\label{Yuck}
Let $U,V \in \X$ and $g \in G$.
Then there exist $\t \in \mathbf{T}_{(+)}$ and $\beta \in R^\star$ such that 
\[
(\1 \we U\we g\cdot C \we g\ol{c}\cdot V,1)=\t(\c,\beta)
\]
for all $\c \in R$.
Dually, there exists $\t \in \mathbf{T}_{(*)}$ and $\beta \in R^\star$ such that
\[
(\1 \we U\we g\ol{c}^{-1}\cdot C \we g\ol{c}^{-1}\cdot V,1)=\t(\c,\beta)
\]
for all $\c \in R$.
\end{Lem}

\begin{proof}
Let $g=w_1\cdots w_n$ be a 
 nice  
factorisation.
We prove the lemma by induction on $n$.

If $n=1$ then there are four cases to consider (recall the definition of the element $(\tilde{C},\tilde{c})$ from Lemma \ref{Onedir}):
\[
\begin{array}{c|c}
\text{Conditions}& ((\1\we U\we w_1 \cdot \tilde{C} \we w_1\tilde{c}\cdot V,1)= \\
\hline
w_1 \in T,\ \t \in \mathbf{T}_+ & \big((\1\we U\we w_1\cdot \1,w_1)(C,\ol{c})(\1\we V,1)\big)^+  \\
w_1 \in T^{-1},\ \t \in \mathbf{T}_+ & \left( \left( (C,\ol{c})(\1 \we V,1)\right)^+ (\1\we w_1^{-1}\cdot U \we w_1^{-1} \cdot \1,w_1^{-1})\right)^* \\
w_1 \in T,\ \t \in \mathbf{T}_*& \left( (\1 \we U\we w_1\cdot \1,w_1) \big((\1 \we V,1)(C,\ol{c})\big)^* \right)^+ \\
w_1 \in T^{-1},\ \t \in \mathbf{T}_*& \big( (\1 \we V,1) (C,\ol{c}) (\1 \we w_1^{-1}\cdot U \we w_1^{-1}\cdot \1,w_1^{-1})\big)^* 
\end{array}
\]
The equalities are easy to check and left to the reader.
Note that throughout variations of the inequalities $w_1 \cdot \1 \geq w_1\cdot C$ and $c \cdot \1 \geq C$ are used.

Let us suppose that for some $n\geq 2$ we have proven the lemma for every $g' \in G$ which admits a 
 nice  
factorisation having less than $n$ factors, and let us suppose that 
$g=w_1\cdots w_n$ is a 
 nice  
factorisation.
Then $w_1^{-1}g=w_2\cdots w_n$ is also a 
 nice  
factorisation, so applying the induction hypothesis for $U'=w_1^{-1}\cdot U,\, V'=V$ and $g'=w_1^{-1}g$, 
there exists $\t' \in \mathbf{T}_{(+)}$ and $\beta' \in R^\star$ such that
\[
(\1 \we w_1^{-1}\cdot U\we w_1^{-1}g\cdot C\we w_1^{-1}g\ol{c}\cdot V,1)=\t'(\c,\beta')
\]
for all $\c \in R$.
Now there are two cases to consider.
If $w_1 \in T$ then $w_1\cdot \1 \geq \1 \we g\cdot \1 \geq \1 \we g \cdot C$, and so
\begin{eqnarray*}
(\1 \we U\we g\cdot C\we g\ol{c}\cdot V,1)\!\!\!&=&\!\!\!\big((\1\we w_1 \cdot \1,w_1) (\1\we w_1^{-1}\cdot U\we w_1^{-1}g\cdot C\we w_1^{-1}g\ol{c}\cdot V,1)\big)^+\cr
\!\!\!&=&\!\!\!\big((\1 \we w_1\cdot \1,w_1) \t'(\c,\beta')\big)^+=\t(\c,\beta)
\end{eqnarray*}
for $\t=(y_i \t')^+$ and $\beta=\big( \beta',(\1 \we w_1\cdot \1, w_1)\big)$.
If $w_1 \in T^{-1}$ then by the same fact we have
\begin{eqnarray*}
(\1 \we U\we g\cdot C\we g\ol{c}\cdot V,1)\!\!\!&=&\!\!\!\big( (\1\we w_1^{-1}\cdot U\we w_1^{-1}g\cdot C\we w_1^{-1}g\ol{c}\cdot V,1)(\1\we w_1^{-1}\cdot \1,w_1^{-1})\big)^*\cr
\!\!\!&=&\!\!\!\big(\t'(\c,\beta')(\1\we w_1^{-1}\cdot \1,w_1^{-1})\big)^*=\t(\c,\beta)
\end{eqnarray*}
for $\t=( \t' z_i)^*$ and $\beta=\big( \beta', (\1 \we w_1^{-1}\cdot \1, w_1^{-1})\big)$.
To finish this part of the proof, one has to note that the fact that the factorisation $g=w_1\cdots w_n$ is alternating implies that the terms defined above are all contained in 
$\mathbf{T}_{(+)}$.
This proves the first statement of the lemma.
The induction step for the second statement is proven similarly, utilising the fact that $w_1 \cdot \1 \geq \1 \we g\ol{c}^{-1} \cdot C$.
\end{proof}

\begin{Lem}\label{Two}
Let $\t \in \mathbf{T}$ and let $\alpha \in (\X \rtimes T)^\star$.
Then there exist $\t' \in \mathbf{T}$ and $\beta \in R^\star$  such that
\begin{equation}\label{Eq}
\t(\c,\alpha)\da=\t'(\c,\beta)
\end{equation}
for all $\c \in R$.
\end{Lem}

\begin{proof}
There are three cases to consider: the first one is if $\t=y x z$.
In this case let $\t'=\t$, and let $\beta=\big( (\1 \we A \we \ol{a}\cdot \1,\ol{a}), (\1 \we B\we \ol{b}\cdot \1,\ol{b})\big)$ 
if $\alpha=\big( (A,\ol{a}),(B,\ol{b})\big)$.
Note that since $\c \in R$, the inequalities $\ol{a}\ol{c}\cdot \1 ,\ol{a}\cdot \1 \geq \ol{a}\cdot C$ imply that
\begin{eqnarray*}
\t(\c,\alpha)\da
\!\!\!&=&\!\!\!(\1 \we A \we \ol{a}\cdot C\we \ol{a}\ol{c}\cdot B \we \ol{a}\ol{c}\ol{b} \cdot \1,\ol{a}\ol{c}\ol{b})\cr
\!\!\!&=&\!\!\!(\1 \we A\we \ol{a}\cdot \1,\ol{a})(C,\ol{c})(\1 \we B\we \ol{b}\cdot \1,\ol{b})=\t'(\c,\beta).
\end{eqnarray*}

The second case is where $\t=y\r z$ for some $\r \in \mathbf{T}_{(+)}$.
Then $\alpha=\big(\alpha', (A,\ol{a}),(B,\ol{b})\big)$ for some $(A,\ol{a}),(B,\ol{b}) \in \X \rtimes T$ and $\alpha' \in (\X \rtimes T)^\star$.
By Lemma \ref{Onedir} there exist $U,V \in \X$ and $g \in G$ such that $\r(\c,\alpha')=(U\we g\cdot C\we g\ol{c} \cdot V,1)$ for every $\c \in R$.
In this case 
\[
\t(\c,\alpha)=(A,\ol{a})(U\we g\cdot C\we g\ol{c}\cdot V,1)(B,\ol{b})=(A \we \ol{a}\cdot U\we \ol{a}g\cdot C\we \ol{a}g\ol{c}\cdot V\we \ol{a}\cdot B,\ol{a}\ol{b}).
\]
By Lemma \ref{Yuck}, there exist $\tilde{\t} \in \mathbf{T}_{(+)}$ and $\beta' \in R^\star$ such that for all $\c \in R$ we have
\[
(\1 \we A \we \ol{a}\cdot U\we \ol{a}g\cdot C\we \ol{a}g\ol{c}\cdot V \we \ol{a}\cdot B,1)=\tilde{\t}(\c,\beta').
\]
So altogether
\begin{eqnarray*}
\t(\c,\alpha)\da\!\!\!&=&\!\!\!(\1\we A \we \ol{a}\cdot U\we \ol{a}g\cdot C\we \ol{a}g\ol{c}\cdot V \we \ol{a}\cdot B \we \ol{a}\ol{b}\cdot \1,\ol{a}\ol{b})\cr
\!\!\!&=&\!\!\!(\1 \we A \we \ol{a}\cdot U\we \ol{a}g\cdot C\we \ol{a}g\ol{c}\cdot V \we \ol{a}\cdot B,1)(\1 \we \ol{a}\ol{b}\cdot \1,\ol{a}\ol{b})\cr
\!\!\!&=&\!\!\!\tilde{\t}(\c,\beta')(\1 \we \ol{a}\ol{b}\cdot \1,\ol{a}\ol{b}),
\end{eqnarray*}
so $\t'=y \tilde{\mathbf{t}} z$ and $\beta=\big( (\1,1),(\1 \we \ol{a}\ol{b}\cdot \1,\ol{a}\ol{b})),\beta'\big)$ satisfy the requirements of the lemma.

The third case, where $\t \in \mathbf{T}_{(*)}$, can be dealt with similarly to the previous case.
\end{proof}

\begin{Lem}\label{Main1}
Let $\rho$ be a $(2,1,1)$-congruence on $R$ and let $\rho^\#$ be the $(2,1,1)$-congruence on $\X \rtimes T$ generated by $\rho$.
Then the restriction of $\rho^\#$ to $R$ equals $\rho$, i.e., $\rho^\#$ extends $\rho$.
As a consequence, $R/\rho$ is $(2,1,1)$-embeddable in $(\X \rtimes T)/\rho^\#$.
\end{Lem}

\begin{proof}
Let $(\s,\t) \in \rho^\# \cap (R \times R)$.
Then there exists a $\rho$-sequence
\[
\s=\t_1(\c_1,\alpha_1),\, \t_1(\d_1,\alpha_1)=\t_2(\c_2,\alpha_2), \ldots,\, \t_n(\d_n,\alpha_n)=\t
\]
connecting $\s$ and $\t$ in $\X \rtimes T$.
By Lemma \ref{Two}, for every $k\ (1\leq k\leq n)$, there exist $\t'_k \in \mathbf{T}$ and $\beta_k \in R^\star$ satisfying $(\ref{Eq})$.
Thus for every $k\ (1\leq k\leq n)$, we have
\[
\t_k(\c_k,\alpha_k)\da=\t'_k(\c_k,\beta_k),\ \t'_k(\d_k,\beta_k)=\t(\d_k,\alpha_k)\da,
\]
showing that the sequence
\[
\s=\t_1(\c_1,\alpha_1)\da=\t'_1(\c_1,\beta_1),\, \t'_1(\d_1,\beta_1)=\t'_2(\c_2,\beta_2),\ldots,\, \t'_n(\d_n,\beta_n)=\t_n(\d_n,\alpha_n)\da=\t
\]
connects $\s$ and $\t$ within $R$, proving that $(\s,\t) \in \rho$.
\end{proof}

Although this lemma easily implies our 
embedding result referred in the title to, before formulating it, we show two lemmas in order to prepare
our result on proper covers.
Note that in the first lemma the action of $T$ need not be nice.

\begin{Lem}\label{in-sigma}
Let $\rho \subseteq \sigma$ be a $(2,1,1)$-congruence on the semidirect product $\X \rtimes T$.
Then, for every $(A,\ol{a}),(B,\ol{b})\in \X\rtimes T$, we have 
$(A,\ol{a})\,\rho\,(B,\ol{b})$ if and only if $\ol{a}=\ol{b}$ and $(A,1)\, \rho \, (B,1)$.
As a consequence, $\rho$ is generated by its restriction to the projections.
\end{Lem}

\begin{proof}
Let $\tau=\{ \big((A,\ol{a}),(B,\ol{b})\big)\in (\X \rtimes T)\times(\X \rtimes T): \ol{a}=\ol{b}\ \hbox{and}\ (A,1)\, \rho \, (B,1)\}$.
It is clear that $\rho \subseteq \tau$.
Conversely, let $\big( (A,\ol{a}),(B,\ol{b})\big) \in \tau$.
Then
\[
(A\we B,\ol{a})=(A,1)(B,\ol{a}) \ \rho\ (B,1)(B,\ol{a})=(B,\ol{a}),
\]
which together with its dual completes the proof.
\end{proof}

\begin{Lem}\label{epsilon}
Let $\rho$ be a $(2,1,1)$-congruence on $R$.
Denote by $\rho_P$ the restriction of $\rho$ to $P=P(R)$ and by $\tau$ and $\tau^\#$ the $(2,1,1)$-congruences on $R$ and on $\X \rtimes T$, 
respectively, generated by $\rho_P$.
Then $R/\tau$ $(2,1,1)$-embeds into $(\X/\varepsilon) \rtimes T$ for an appropriate $G$-invariant congruence $\varepsilon$ on $\X$,
and $(\X/\varepsilon) \rtimes T$ is a 
 nice  
semidirect product of $(\X/\varepsilon)$ by $T$. 
Moreover, the $(2,1,1)$-congruence $\rho/\tau$ on $R/\tau$ is projection separating, and the corresponding congruence on the image of $R/\tau$
in $(\X/\varepsilon) \rtimes T$ extends to a  $(2,1,1)$-congruence on $(\X/\varepsilon) \rtimes T$.
Consequently, $R/\rho$ is $(2,1,1)$-embeddable in a $(2,1,1)$-morphic image of $(\X/\varepsilon) \rtimes T$.
\end{Lem}

\begin{proof}
Notice that $\tau\subseteq\sigma_R$ and $\tau^\#\subseteq\sigma_{\X\rtimes T}$.
By definition, $\rho_P \subseteq \tau \subseteq \rho$, and as the restriction of $\rho$ to $P$ equals $\rho_P$  the same is true for $\tau$, thus the $(2,1,1)$-congruence $\rho/\tau$ on $R/\tau$ is projection separating.
Denote by $\rho^\#$ the $(2,1,1)$-congruence on $\X\rtimes T$ generated by $\rho$.
Since $T$ acts 
 nicely over $G$ on $\X$, we see by Lemma \ref{Main1} that $\rho^\#$ and $\tau^\#$ extend $\rho$ and $\tau$, respectively.
Put $\varepsilon=\{(A,B)\in \X\times \X : (A,1)\,\tau^\#\,(B,1)\}$.
Then $\varepsilon$ is obviously a congruence on $\X$, and it is easily seen to be $G$-invariant.
For, let $A,B\in\X$ with $A\,\varepsilon\,B$ and $\ol a\in T$.
Then we obtain by Lemma \ref{in-sigma} that $(A,\ol a)\,\tau^\#\,(B,\ol a)$ which implies
$({\ol a}^{-1}\cdot A,1)=(A,\ol a)^*\,\tau^\#\,(B,\ol a)^*=({\ol a}^{-1}\cdot B,1)$.
Hence ${\ol a}^{-1}\cdot A\, \varepsilon\,
{\ol a}^{-1}\cdot B$ follows. 
Moreover, we also see that
$(\ol a\cdot A,1)=(\ol a\cdot A,\ol a)^+=\big((\ol a\cdot A,\ol a)(A,1)\big)^+\,\tau^\#\,
\big((\ol a\cdot A,\ol a)(B,1)\big)^+=(\ol a(A\wedge B),\ol a)^+=(\ol a(A\wedge B),1)$,
and so we deduce that $\ol a\cdot A\;\varepsilon\;\ol a\cdot(A\wedge B)$.
Similarly,  $\ol a\cdot B\;\varepsilon\;\ol a\cdot(A\wedge B)$ also holds, implying the relation $\ol a\cdot A\;\varepsilon\;\ol a\cdot B$.
Since $G$ is generated by $T$, one can see that $g\cdot A\;\varepsilon\;g\cdot B$ for every $g\in G$.
Thus $\varepsilon$ is, indeed, $G$-invariant, so the action of $G$ on $\X$ induces an action of $G$ on $\X/\varepsilon$, 
and defines semidirect products $(\X/\varepsilon)\rtimes G$ and $(\X/\varepsilon)\rtimes T$.
Lemma \ref{per-eps} implies that the latter semidirect product is nice.

By Lemma \ref{in-sigma} we also have that $\iota\colon(\X \rtimes T) / \tau^\# \to (\X/\varepsilon) \rtimes T,\ (A,\ol a)\tau^\#\mapsto (A\varepsilon,\ol a)$
is an isomorphism.
Let $\tilde{\rho}$ be the congruence on $(\X/ \varepsilon) \rtimes T$ corresponding via $\iota$ to the congruence $\rho^\#/ \tau^\#$ on $(\X \rtimes T)/\tau^\#$.
The rest follows by basic universal algebra: since $\rho^\#$ and $\tau^\#$ extend $\rho$ and $\tau$, respectively,
we obtain that $\rho^\# / \tau^\#$ extends $\rho / \tau$, thus 
$(2,1,1)$-embedding
$R/\rho$, which is isomorphic to $(R/\tau)/(\rho/\tau)$, into $((\X/\varepsilon) \rtimes T) / \tilde{\rho}$.
\end{proof}

Now we are ready to formulate the main result of the paper.
It strengthens \cite[Theorem 4.1]{M2013} which  establishes that every restriction semigroup is $(2,1,1)$-embeddable 
in an almost left factorisable restriction semigroup, and \cite[Theorem 4.11]{M2012} which says that
every restriction semigroup has a proper cover $(2,1,1)$-embeddable in a $W$-product of a semilattice by a monoid.

Recall from Subsection \ref{sd-f} that a restriction semigroup (monoid) is almost factorisable (factorisable) if and only if it is a $(2,1,1)$-morphic image of a 
semidirect product of a semilattice (semilattice with identity) by a 
monoid.

\begin{Thm}\label{main}
\begin{enumerate}
\item[(i)]
Every restriction semigroup is $(2,1,1)$-embeddable in an almost factorisable restriction semigroup, or, equivalently, in a factorisable restriction monoid.
\item[(ii)]
Every restriction semigroup has a proper cover which is 
$(2,1,1)$-embeddable in a nice semidirect product of a semilattice by a 
monoid, and consequently, in a semidirect product of a semilattice by a group.
\end{enumerate}
\end{Thm}

\begin{proof}
(i)
Let $S$ be a restriction semigroup.
By Lemma \ref{new}, it suffices to embed the restriction monoid $S^1$ into
a $(2,1,1)$-morphic image of a 
semidirect product of a semilattice by a 
monoid.
Clearly, $S^1$ is isomorphic to $\FR / \rho$ for some set $\Omega$ and $(2,1,1)$-congruence $\rho$ on $\FR$.
Here $\FR$ is just the $(2,1,1)$-subsemigroup 
$R$ in (\ref{R}) of the semidirect product $\X\rtimes \Omega^*$ given in Subsection \ref{free-restr}.
Example \ref{Ex1}
shows that the action of $\Omega^*$ on $\X$ is nice over $\FG$ with respect to $\1$, and so Lemma \ref{Main1} implies that
$\FR / \rho$ and its 
$(2,1,1)$-isomorphic copy $S^1$ $(2,1,1)$-embed into
a $(2,1,1)$-morphic image of $\X \rtimes \Omega^*$.

(ii)
Applying Lemma \ref{epsilon} for $R=\FR$ and the $(2,1,1)$-congruence $\rho$ in the previous paragraph,
we obtain that $\FR/\tau$ is $(2,1,1)$-embeddable in a nice semidirect product $\X' \rtimes \Omega^*$ where $\X'$ is a factor semilattice of $\X$
over a $\FG$-invariant congruence.
This implies that the semidirect product $\X' \rtimes \FG$ is also defined, and $\X' \rtimes \Omega^*$ is a $(2,1,1)$-subsemigroup in  $\X' \rtimes \FG$.
The restriction monoid $\FR/\tau$ is proper since the semidirect product $\X' \rtimes \Omega^*$ is.
Moreover, Lemma \ref{epsilon} also establishes that $\rho/\tau$ is a projection separating $(2,1,1)$-congruence on $R/\tau$.
This shows that $\FR/\tau$ is a proper cover of $(\FR/\tau)/(\rho/\tau)$, and so of $\FR/\rho$ and $S^1$.
This completes the proof if $S=S^1$, that is, if $S$ is a monoid.
In the opposite case, it is routine to check that we can ensure the $\rho$-class of the identity element of $\FR$ is a singleton, and the same property follows for $\tau$.
This implies that the restriction semigroup  obtained from $\FR/\tau$ by deleting its identity element is a proper cover of $S$ with the required properties.
\end{proof}

Note that, similarly to the embedding theorems in \cite{M2013} and \cite{K2014}, the embedding given here is not a monoid embedding.


\section{Additional remarks}

One can see from Lemmas \ref{in-sigma} and \ref{epsilon}
(see also the proof of Theorem \ref{main}(ii))
that if $C$ is a restriction
monoid which is a $(2,1,1)$-factor
monoid of a free restriction monoid over a $(2,1,1)$-congruence contained in $\sigma$, 
then $C$ is $(2,1,1)$-embeddable in a nice semidirect product.
Both Kudryavtseva \cite{K2014} and Jones \cite{Jones} have introduced a somewhat more general class of restriction monoids, called 
ultra $F$-restriction monoids in \cite{K2014} and perfect restriction monoids in \cite{Jones}.
Moreover, if the greatest reduced factor of such a restriction monoid is a free monoid then it is proved to be 
$(2,1,1)$-embeddable in a $W$-product in \cite[Theorem 28]{K2014}.
In this section we strengthen this result by considering an even more general class of restriction monoids and showing that their members 
are 
$(2,1,1)$-embeddable in a nice  
semidirect product in such a way that each of their $(2,1,1)$-congruences extends to the semidirect product.
Later on, we adopt the name `perfect' from \cite{Jones} for the class of restriction monoids 
mentioned, and we use the notation of \cite{K2014}. 

Let $T$ be a monoid and $\Y$ a semilattice with identity element.
Denote by $O_\Y$ the inverse submonoid of the symmetric inverse monoid $I(\Y)$ consisting of all isomorphisms between the ideals of $\Y$.
Notice that the Munn semigroup $T_\Y$ is an inverse submonoid in $O_\Y$.
A {\em left partial action of $T$ on $\Y$} is defined to be a dual monoid prehomomorphism $\alpha\colon T\to \op{O}_\Y, t\mapsto \alpha_t$, that is,
$\alpha$ is required to have the properties that $\alpha_1$ is the identity automorphism of $\Y$, and $\alpha_t\alpha_u$ is a restriction
of $\alpha_{tu}$ for every $t,u\in T$.
For any $t\in T$ and $A\in\Y$, we denote the element $\alpha_tA\in\Y$ by $t\diamond A$.
Since $O_\Y$ is an inverse monoid, the mapping $\alpha'\colon T\to O_\Y,\ t\mapsto \alpha_t^{-1}$ 
is a dual monoid prehomomorphism, and so $\alpha'$ defines a right partial action of $T$ on $\Y$.
For any $t\in T$ and $A\in\Y$, the element $A\alpha'_t\in\Y$ will be denoted by $A\circ t$.
Clearly, the partial actions $\diamond$ and $\circ$ are reverse to each other, that is, for every $t\in T$ and $A\in\Y$,
if $t\diamond A$ is defined then $(t\diamond A)\circ t$ is also defined and $(t\diamond A)\circ t=A$, and similarly,
if $A\circ t$ is defined then $t\diamond (A\circ t)$ is also defined and $t\diamond (A\circ t)=A$.

Consider the set
\[M(T,\Y)=\{(A,t)\in \Y\times T : A\circ t\ \hbox{is defined}\},\]
and define the following operations on it:
\[(A,t)(B,u)=(t\diamond((A\circ t)\wedge B),tu),\quad
(A,t)^+=(A,1)\quad\hbox{and}\quad
(A,t)^*=(A\circ t,1).\]
Then $M(T,\Y)$ is a proper restriction monoid, 
and, conversely, each proper restriction monoid
$S$ is isomorphic to $M(S/\sigma,P(S))$ where the left partial action of $S/\sigma$ on $P(S)$ is induced by $S$ in a natural way, see \cite{CG}.

A congruence $\rho$ on a semigroup $S$ is called {\it perfect} if for every $a,b \in S$ we have 
$(a\rho)(b\rho)=(ab)\rho$, where 
the left hand side of the equality is the set product of the classes $a\rho$ and $b\rho$.
A restriction monoid $S$ is called {\em perfect} if 
$\sigma$ is a perfect congruence and each $\sigma$-class has a greatest element with respect to the natural partial order.
It follows (\cite[Theorem 1.1]{Jones})
that a restriction monoid is perfect if and only if it is isomorphic to a restriction monoid $M(T,\Y)$
where the left partial action of $T$ on $\Y$ is a monoid homomorphism from $T$ into $\op{T}_\Y$.

\begin{Rem}
\begin{enumerate}
\item[(i)] 
Note that \cite{K2014} uses the term `left partial action' in a more general sense.
However, in a semilattice, the order ideals (resp.\ principal order ideals) and the ideals (resp.\ principal ideals) coincide.
Moreover, the order isomorphisms and the isomorphisms between the ideals of a semilattice coincide.
Therefore a left partial action of a monoid $T$ on a semilattice $\Y$ defined above is just what is called in \cite{K2014} a 
`left partial action of $T$ on $\Y$ such that axioms (A),(B),(C) hold'.
\item[(ii)]
The restriction monoid denoted by $W(T,\Y)$ in \cite{Jones} is the left-right dual of $M(T,\Y)$ defined above,
and so it is a much more general construction than what is usually meant by the notation  
$W(T,\Y)$ (see \cite{FG1=5}, \cite{GSz}, \cite{K2014}, \cite{M2012}, \cite{M2013}), 
in particular in this paper, called a $W$-product of a semilattice by a monoid and denoted $W(T,\Y)$.
\end{enumerate}
\end{Rem}

Suppose now that $T$ is a monoid embeddable in a group, $\X$ is a semilattice, and $T$ acts on $\X$ by automorphisms.
As above, let us choose and fix a group $G$ such that 
$T$ is a submonoid in $G$ generating $G$, $G$ acts on $\X$, 
and the action of $T$ on $\X$ is just the restriction of the action of $G$ on $\X$.
Moreover, consider a principal ideal $\Y$ of $\X$, and denote its greatest element by $\1$.
It is routine to check that the restriction of the action of $G$ on $\X$ to $\Y$ is a left partial action
$\alpha\colon G\to \op{T}_\Y$ where the domain and range of $\alpha_g\ (g\in G)$ are the principal ideals of $\Y$ with greatest elements
$g^{-1}\cdot\1\wedge\1$ and $g\cdot\1\wedge\1$, respectively.
Therefore the restriction monoids $M(G,\Y)$ and $M(T,\Y)$ are defined, and we have 
$M(T,\Y)\leqslant M(G,\Y)$.
Taking into account that all elements of $\X$ appearing in $M(G,\Y)$ 
belong to the subsemilattice $G\cdot\Y$, 
we can suppose without loss of generality that $\X=G\cdot\Y$.
In this case, $(G,\X,\Y)$ is a McAlister triple, and it is easy to see that
$M(G,\Y)=P(G,\X,\Y)$, the $P$-semigroup defined by it.
This implies that $M(G,\Y)$ is an $F$-inverse monoid, and it is an inverse subsemigroup in the semidirect poduct $\X\rtimes G$.
Furthermore, it also follows that $M(T,\Y)$ is equal to the $(2,1,1)$-subsemigroup $R$ of the semidirect product $\X\rtimes T$
defined in (\ref{R}).
Thus Lemmas \ref{Main1} and \ref{epsilon} imply the following statement.

\begin{Prop}\label{K1}
Suppose that a left partial action of a monoid $T$ on a semilattice $\Y$ with identity $\1$ can be extended to 
an action $\cdot$ of $T$ on a semilattice $\X$ containing $\Y$ as a principal ideal such that 
the action $\cdot$ is 
 nice  
over a group $G$ with respect to $\1$, and we have $\X=G\cdot\Y$.
Then the following hold:
\begin{enumerate}
\item[(i)] $M(T,\Y)\leqslant \X\rtimes T,  M(G,\Y)\leqslant \X\rtimes G$ where $\X\rtimes T$ is a 
 nice  
semidirect product
and $M(G,\Y)$ is an $F$-inverse monoid;
\item[(ii)] every $(2,1,1)$-congruence $\rho$ on $M(T,\Y)$ extends to $\X\rtimes T$;
\item[(iii)] for every $(2,1,1)$-congruence $\rho$ of $M(T,\Y)$ with $\rho\subseteq\sigma_{M(T,\Y)}$,
the restriction monoid $M(T,\Y)/\rho$ embeds into a 
 nice  
semidirect product of a factor semilattice of $\X$ by $T$.
\end{enumerate}
\end{Prop}

Since any proper restriction monoid $S$ is $(2,1,1)$-isomorphic to the monoid $M(S/\sigma,P(S))$ defined by means of 
the induced left partial action, we have the following consequence.

\begin{Cor}\label{K2}
Let $S$ be a proper restriction monoid such that the induced left partial action of $S/\sigma$ on $P(S)$ can be extended 
to an action $\cdot$ of $S/\sigma$ on a semilattice $\X$ containing $P(S)$ as a principal ideal in such a way that
the action $\cdot$ is 
 nice  
over a group $G$ with respect to the identity element of $P(S)$ and $\X=G\cdot P(S)$. 
Then $S$ is $(2,1,1)$-embeddable in the 
 nice  
semidirect product $\X\rtimes (S/\sigma)$,
which is a $(2,1,1)$-subsemigroup in the inverse semigroup $\X\rtimes G$, and the 
$(2,1,1)$-embedding  has the following properties:
\begin{enumerate}
\item[(i)] each $(2,1,1)$-congruence of $S$ extends to $\X\rtimes (S/\sigma)$ (along the 
$(2,1,1)$-embedding), and
\item[(ii)] for every $(2,1,1)$-congruence $\rho$ of $S$ with $\rho\subseteq\sigma$,
the restriction monoid $S/\rho$ embeds into a 
 nice  
semidirect product of a factor semilattice of $\X$ by $S/\sigma$.
\end{enumerate}
\end{Cor}

Now we establish that the perfect restriction monoids $M(\Omega^*,\Y)$ studied in \cite{K2014} and \cite{Jones} satisfy the assumptions
of Proposition \ref{K1} and Corollary \ref{K2}.

Let $\Omega$ be a set, $\Y$ a semilattice with identity element $\1$, and consider a left partial action of
the free monoid $\Omega^*$ on $\Y$ which is a monoid homomorphism $\alpha\colon \Omega^*\to \op{T}_\Y$.
This defines a perfect restriction monoid $S=M(\Omega^*,\Y)$.
The free monoid $\Omega^*$ is a submonoid of the free group $\FG$ such that $\FG$ is generated by $\Omega^*$, and
each element of $\FG$ can be uniquely written as a product of members of an alternating sequence, see Example \ref{Ex1}.
This allows us to extend $\alpha$ to a left partial action, also denoted by $\alpha$, of $\FG$ on $\Y$ by
putting $\alpha_{t^{-1}}=\alpha'_t$ for every $t\in\Omega^*$, and then by
defining $\alpha_{w_1w_2\cdots w_n}=\alpha_{w_1}\alpha_{w_2}\cdots\alpha_{w_n}$ in $\op{T}_\Y$ for every alternating sequence $w_1,w_2,\ldots,w_n$ in $\Omega^*$.
Notice that the left partial action $\alpha\colon\FG\to \op{T}_\Y$ is not a monoid homomorphism any more
(cf.\ \cite[arXiv:1402.5849v2, p.\ 23]{K2014}).
However, if $g,h\in\FG$ such that $g,h$ and $gh$ (the concatenation of $g$ and $h$) are reduced words then $\alpha_g\alpha_h=\alpha_{gh}$.
For every $g\in\FG$ denote the identity element of the range of $\alpha_g$ by $M_g$.
By definition, we clearly have $M_1=\1$ and, since $M_{g^{-1}}$ equals the identity element of 
the domain of $\alpha_g$, we have
\begin{equation}\label{Mg}
M_g=g\diamond M_{g^{-1}}\quad \hbox{for every}\ g\in\FG.
\end{equation}

The restriction monoid $M(\FG,\Y)$ is easily seen to be an $F$-inverse monoid,
where the greatest element in the $\sigma$-class corresponding to $g\in\FG$ is $(M_g,g)$.
Hence $M(\FG,\Y)$ can be represented as a $P$-semigroup
$P(\FG,\Y,\X)$ where $\X$ is a semilattice and is, up to isomorphism, uniquely determined.
Futhermore, by \cite{Munn}, $\X$ can be constructed in the following manner.
The relation $\chi$ on the set $\Y\times\FG$, defined by $(A,g)\,\chi\,(B,h)$ if and only if $B=A\circ g^{-1}h$,
is a preorder, and so it induces a partial order $\le$ on the set
$\X$ of all $\overline{\chi}$-classes where $\overline{\chi}$ is the join of $\chi$ and its converse.
Denoting the $\overline{\chi}$-class of the element $(A,g)\in \Y\times \FG$ by $[A,g]$, we can deduce from the proof of
\cite[Theorem 7.4.5]{L=8} that
\begin{equation}\label{meet}
[A,g]\wedge[B,h]=\left[g^{-1}h\diamond\left((A\wedge M_{g^{-1}h})\circ g^{-1}h\wedge B\right),g\right]
\end{equation}
for every element $[A,g],[B,h]$ in the semilattice $\X$.
Hence $\X_0=\{[A,1]:A\in \Y\}$ is a principal ideal in $\X$, and the mapping $\Y\to\X_0,\ A\mapsto [A,1]$ is an isomorphism.
By identifying $\X_0$ with $\Y$ along this isomorphism, $\Y$ becomes a principal ideal of $\X$.
The rule $h\cdot [A,g]=[A,hg]\ (h\in\FG,\ [A,g]\in\X)$ defines an action of the group $\FG$ on $\X$ such that $\X=\FG\cdot \Y$, and
the restriction of this action to $\Y$ is obviously the left partial action $\diamond$.

On the other hand, the action $\cdot$ of $\FG$ on $\X$ naturally restricts to an action, also denoted by $\cdot$, of $\Omega^*$ on $\X$.
Now we establish that this action of $\Omega^*$ on $\X$ is 
 nice  
over $\FG$ with respect to the element $\1$.
It suffices to show that if $g,h\in\FG$ are reduced words such that there is no non-empty common prefix of $g$ and $h$
then the inequality $\1\ge g\cdot\1\wedge h\cdot\1$ holds in $\X$, or equivalently, $g\cdot\1\wedge h\cdot\1\in\Y$.
Applying (\ref{meet}), (\ref{Mg}) and that $M_g\in\Y$ for every $g\in\FG$, we obtain that 
\begin{eqnarray*}
[\1,g]\wedge [\1,h]\!\!\!&=&\!\!\!\left[g^{-1}h\diamond\left((\1\wedge M_{g^{-1}h})\circ g^{-1}h\wedge \1\right),g\right]\cr
\!\!\!&=&\!\!\!\left[g^{-1}h\diamond(M_{g^{-1}h}\circ g^{-1}h),g\right]=[M_{g^{-1}h},g]=[g^{-1}h\diamond M_{h^{-1}g},g].\cr
\end{eqnarray*}
Since $g^{-1}h$ is a reduced word by assumption, we have
$g^{-1}h\diamond M_{h^{-1}g}=\alpha_{g^{-1}h} M_{h^{-1}g}=\alpha_{g^{-1}}(\alpha_h M_{h^{-1}g})$,
whence it follows that both $\alpha_h M_{h^{-1}g}=h\diamond M_{h^{-1}g}$ and $\alpha_{g^{-1}}(\alpha_h M_{h^{-1}g})=(h\diamond M_{h^{-1}g})\circ g$ are defined.
This implies by the definition of the preorder $\chi$ that
$$[\1,g]\wedge [\1,h]=[(h\diamond M_{h^{-1}g})\circ g,g]=[h\diamond M_{h^{-1}g},1]\in\Y,$$
thus completing the proof that $g\cdot\1\wedge h\cdot\1\in\Y$.

Summarising, we have strengthened the results \cite[Theorem 28]{K2014} and \cite[Proposition 8.4]{Jones} in the following way:

\begin{Cor}
Let $S$ be a perfect restriction monoid such that $S/\sigma$ is a free monoid on a set $\Omega$.
Then the induced left partial action of $S/\sigma$ on $P(S)$ can be extended 
to an action $\cdot$ of $S/\sigma$ on a semilattice $\X$ contaning $P(S)$ as a principal ideal
such that the action $\cdot$ is 
 nice  
over the free group $\FG$ with respect to the identity element of $P(S)$ 
and $\X=\FG\cdot P(S)$.
Consequently, $S$ is 
$(2,1,1)$-embeddable in the 
 nice  
semidirect product $\X\rtimes \Omega^*$,
which is a $(2,1,1)$-subsemigroup in the inverse semigroup $\X\rtimes \FG$, and the following are valid:
\begin{enumerate}
\item[(i)] each $(2,1,1)$-congruence of $S$ extends to $\X\rtimes \Omega^*$ (along the 
$(2,1,1)$-embedding), and
\item[(ii)] for every $(2,1,1)$-congruence $\rho$ of $S$ with $\rho\subseteq\sigma$,
the restriction monoid $S/\rho$ embeds into a 
 nice  
semidirect product of a factor semilattice of $\X$ by $\Omega^*$.
\end{enumerate}
\end{Cor}

\end{document}